\documentclass{amsart}
\title{A Tensor-Cube Version of the Saxl Conjecture}
\author{Nate Harman}
\author{Christopher Ryba}
\address[Nate Harman]{Department of Mathematics, University of Michigan, Ann Arbor, MI, 48101, USA}
\email{nharman@umich.edu}
\address[Christopher Ryba]{Department of Mathematics, University of California, Berkeley, CA 94720, USA}
\email{ryba@math.berkeley.edu}
\usepackage{hyperref}
\usepackage{fullpage}
\usepackage{amsmath}
\usepackage{amsfonts}
\usepackage{amssymb}
\usepackage{ytableau}
\usepackage{subfig}
\usepackage{float}
\usepackage{graphicx, caption}
\usepackage{comment}
\usepackage{xcolor}

\definecolor{col1}{HTML}{d7191c}
\definecolor{col2}{HTML}{fdae61}
\definecolor{col3}{HTML}{4bbd84}
\definecolor{col4}{HTML}{2b43ba}

\begin{document}
\maketitle
\begin{abstract}
Let $n$ be a positive integer, and let $\rho_n = (n, n-1, n-2, \ldots, 1)$ be the ``staircase'' partition of size $N = {n+1 \choose 2}$. The Saxl conjecture asserts that every irreducible representation $S^\lambda$ of the symmetric group $S_N$ appears as a subrepresentation of the tensor square $S^{\rho_n} \otimes S^{\rho_n}$. In this short note we give two proofs that every irreducible representation of $S_N$ appears in the tensor cube $S^{\rho_n} \otimes S^{\rho_n} \otimes S^{\rho_n}$.
\end{abstract}
\newtheorem{theorem}{Theorem}
\newtheorem{lemma}[theorem]{Lemma}
\newtheorem{proposition}[theorem]{Proposition}
\newtheorem{corollary}[theorem]{Corollary}
\newtheorem{definition}[theorem]{Definition}
\newtheorem{example}[theorem]{Example}
\newtheorem{remark}[theorem]{Remark}
\newtheorem{observation}[theorem]{Observation}

\section{Introduction}
\noindent
Let $n, \rho_n, N$ be as defined in the abstract. In 2012, Jan Saxl conjectured that $S^{\rho_n} \otimes S^{\rho_n}$ contains every irreducible representation of $S_N$ as a subrepresentation. This was motivated by a similar phenomenon for finite simple groups of Lie type \cite{HSTZ}. Despite a great amount of effort, tensor products of symmetric group representations are poorly understood from a combinatorial point of view. The Saxl conjecture serves as one benchmark for results in this area. Progress has been made using a wide variety of approaches, for example, using ideas from combinatorics \cite{PPV} \cite{Li}, modular representation theory \cite{BBS} \cite{BB}, or probability \cite{Sellke} \cite{LS}. In this last paper mentioned, Luo and Sellke proved a weaker version of the Saxl conjecture, namely that for $n$ sufficiently large, $\left(S^{\rho_n}\right)^{\otimes 4}$ contains every irreducible representation of $S_N$ as a subrepresentation.  Our main theorem is the following improvement:

\begin{theorem}\label{thm:main}
 $\left(S^{\rho_n}\right)^{\otimes 3}$ contains all irreducible representations of $S_N$.
\end{theorem}

 We note that the results of Luo and Sellke extend to $N$ not of the form ${n+1 \choose 2}$, while ours rely on special properties of staircase partitions which do not seem to generalize. We provide two proofs, the first using combinatorics, and the second using modular representation theory.

\section*{Acknowledgements}
\noindent
The first author would like to thank Joshua Mundinger for helpful conversations. The second author would like to thank Mark Sellke for helpful comments on a previous version of this paper.

\section{Proof using Combinatorics}
\noindent
For now, let $N$ be an arbitrary positive integer, which we will specialise to ${n+1 \choose 2}$ at the end of this section. We use the following result of Luo and Sellke (although for our final conclusion the prior result of Ikenmeyer \cite{Ikenmeyer} is sufficient).

\begin{lemma}[Luo-Sellke, Appendix B \cite{LS}] \label{thm:LuoSellke}
Let $\mu, \nu$ be partitions of size $N$ such that $\mu$ has distinct parts, and $\nu$ is greater than or equal to $\mu$ in the dominance order. Then $S^{\mu} \otimes S^{\mu}$ contains $S^{\nu}$ as a subrepresentation.
\end{lemma}
\noindent
If $\mu = (\mu_1, \mu_2, \ldots, \mu_r)$ is a partition of $N$, let $M^{\mu}$ be the permutation representation on cosets of the Young subgroup $S_{\mu} = S_{\mu_1} \times S_{\mu_2} \times \cdots \times S_{\mu_r}$ of $S_N$. Equivalently, $M^{\mu} = \mathrm{Ind}_{S_{\mu}}^{S_N} \left( \mathbf{1} \right)$, where $\mathbf{1}$ is the trivial representation.

\begin{definition}
Given a partition $\mu$, let $C(\mu)$ be the partition obtained as follows. Start off with the empty partition. For each $i \geq 1$, add $i$ parts of size $\mu_i - \mu_{i+1}$ (if $i$ is greater than the length of $\mu$, we define $\mu_i = 0$).
\end{definition}

\begin{example} \label{exmpl:staircase}
Suppose that $\mu = \rho_n$ is the staircase partition of size ${n+1 \choose 2}$. Then $\mu_i - \mu_{i+1} = 1$ for $i = 1,2, \ldots, n$. Thus $C(\mu)$ consists of the part $1$ repeated $1 + 2 + \cdots + n = {n+1 \choose 2}$ times.
\end{example}

\begin{theorem} \label{thm:tensor_summand}
Let $\mu$ be any partition. Then $M^{\mu} \otimes S^{\mu}$ contains $M^{C(\mu)}$ as a subrepresentation.
\end{theorem}

\begin{proof}
It is convenient to take Frobenius characteristics to turn representations into symmetric functions. The Frobenius characteristics of $M^{\mu}$ and $S^\mu$ are $h_\mu$ and $s_\mu$ respectively. Their tensor product is described by the internal product of symmetric functions, which admits the following description (see Example 23(c) of Section 1.7 of \cite{Macdonald}):
\[
h_\mu * s_\mu = \sum \prod_{i\geq 1} s_{\lambda^{(i)} / \lambda^{(i+1)}}
\]
where the sum ranges over all nested families of partitions $\lambda^{(i)}$ obeying $\mu = \lambda^{(1)} \supset \lambda^{(2)} \supset \cdots \supset \varnothing$, such that $|\lambda^{(i)}| - |\lambda^{(i+1)}| = |\mu_i|$.
\newline \newline \noindent
Each summand is the Frobenius characteristic of a genuine (rather than virtual) representation, so to show that $M^{C(\mu)}$ is a subrepresentation of $M^{\mu} \otimes S^\mu$, it suffices to show that one of the summands is equal to $h_{C(\mu)}$. We do this by considering the term where $\lambda^{(i)}$ is obtained from $\mu$ by removing the first $(i-1)$ rows (so $\lambda_j^{(i)} = \mu_{j+i-1}$). Then it is not difficult to see that the skew-diagram $\lambda^{(i)} / \lambda^{(i+1)}$ consists of the disjoint union of rows of lengths $\mu_j - \mu_{j+1}, \mu_{j+1} - \mu_{j+2}, \ldots$. We illustrate this with an example.
\newline \newline \noindent
Suppose that $\mu = (10,6,4,1)$. Then we have $\lambda^{(1)} = (10,6,4,1)$, $\lambda^{(2)} = (6,4,1)$, $\lambda^{(3)} = (4,1)$, $\lambda^{(4)} = (1)$. In the following diagram, $\lambda^{(1)} / \lambda^{(2)}$ is depicted in {\color{col1}red}, $\lambda^{(2)} / \lambda^{(3)}$ is depicted in {\color{col2}orange}, $\lambda^{(3)} / \lambda^{(4)}$ is depicted in {\color{col3}green}, and $\lambda^{(4)}$ is depicted in {\color{col4}blue}:
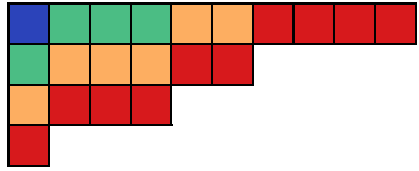
\begin{figure}[H]
\centering
\ytableausetup{nosmalltableaux}
\begin{ytableau}
*(col4) &*(col3) &*(col3) &*(col3) &*(col2) &*(col2) &*(col1) &*(col1) &*(col1) &*(col1) \\
*(col3) & *(col2) & *(col2) &*(col2) &*(col1) &*(col1) \\
*(col2) &*(col1) &*(col1) &*(col1) \\
*(col1)
\end{ytableau}
\end{figure}
\noindent
In particular a horizontal row of size $\mu_1-\mu_2=4$ appears once, a horizontal row of size $\mu_2-\mu_3=2$ appears twice, a horizontal row of size $\mu_3-\mu_4=3$ appears three times, and a horizontal row of size $\mu_4=1$ appears 4 times.
\newline \newline \noindent
Since a skew-Schur function for a skew-diagram $D$ is the product of the skew-Schur functions associated to the connected components of $D$, and the skew Schur function associated to a horizontal row of size $r$ is $s_{(r)} = h_r$, we conclude that $s_{\lambda^{(i)} / \lambda^{(i+1)}} = h_{\mu_i - \mu_{i+1}} \cdot h_{\mu_{i+1} - \mu_{i+2}}\cdot \cdots$. Thus
\[
\prod_i s_{\lambda^{(i)} / \lambda^{(i+1)}} = h_{C(\mu)},
\]
and so we have found the required summand.
\end{proof}

\begin{corollary} \label{cor:general_constituents}
Let $\mu$ be a partition of size $N$ with distinct parts. Then for each $\lambda$ greater than or equal to $C(\mu)$ in the dominance order, $S^\mu \otimes S^\mu \otimes S^\mu$ contains $S^\lambda$ as a summand.
\end{corollary}
\begin{proof}
This follows from Theorem \ref{thm:tensor_summand} and the well known fact that the irreducible constituents of $M^\nu$ are precisely the $S^\lambda$ with $\lambda$ greater than or equal to $\nu$ in the dominance order on partitions.
\end{proof}

\noindent \emph{First Proof of Theorem \ref{thm:main}:}
By Corollary \ref{cor:general_constituents}, $S^{\rho_n} \otimes S^{\rho_n} \otimes S^{\rho_n}$ contains $S^\lambda$ for every partition $\lambda$ greater than or equal to $C(\rho_n) = (1^N)$ (see Example \ref{exmpl:staircase}). However, every partition of $N$ is greater than or equal to $(1^N)$ in the dominance order. \hfill $\square$

\section{Proof using Modular Representation Theory}
\noindent
We will now give a second proof of the main theorem, this time following the approach laid out in \cite{BBS} to study the Saxl conjecture via the $2$-modular representation theory of symmetric groups. The proof itself will be quite short, but in the interest of self-containment and readability we'll first state the facts from modular representation theory we will be using.
\newline \newline \noindent
First we will recall some general facts about projective objects in modular representation theory of a finite group.  This are all fairly standard results in modular representation theory and can be found in many places, but as a standard reference  we'll refer to \cite{Serre}.

\begin{lemma}[See \cite{Serre} chapters 14 and 15]\label{thm:projectives}
Let $G$ be a finite group, $k =\bar{k}$ be an algebraically closed field of characteristic $p > 0$, and let $Rep_p(G)$ denote the category of finite dimensional representations of $G$ over $k$.
\begin{enumerate}
\item The number of isomorphism classes projective indecomposable objects in $Rep_p(G)$ is equal to the number simple objects $Rep_p(G)$, which is equal to the number of conjugacy classes of elements of $G$ with order prime to $p$.  Moreover, each simple object occurs uniquely as an irreducible quotient of a projective indecomposable object.

\item If $P$ is a projective object of $Rep_p(G)$, and $V$ is any any representation then $P \otimes V$ is again a projective object.  Moreover if $P$ is a fixed projective object, then every indecomposable projective object $P'$ arises as a direct summand in a module of the form $P \otimes X$ where $X$ is simple.

\item If $P$ is a projective object in $Rep_p(G)$, then $P$ lifts to characteristic zero -- more precisely $P$ is isomorphic to $\tilde{P} \otimes_\mathcal{O} \mathcal{O}/\mathfrak{m}\mathcal{O}$ where $\mathcal{O}$ denotes the ring of Witt vectors for $k$, $\mathfrak{m}$ is the unique maximal ideal in $\mathcal{O}$, and $\tilde{P}$ is projective in the category of representations of $G$ over $\mathcal{O}$.

\item Every ordinary irreducible representation of $G$ over $K = \text{frac}(\mathcal{O})$ occurs inside $\tilde{P} \otimes_\mathcal{O} K$ for some projective indecomposable representation $\tilde{P}$ of $G$ over $\mathcal{O}$.

\end{enumerate}
\end{lemma}

\noindent
Next we will recall some basic facts and notation specific to the modular representation theory of symmetric groups. 

\begin{lemma}
[See \cite{James} chapter 6]\label{thm:symmodular}
 \ 
\begin{enumerate}
\item Simple objects in $Rep_p(S_n)$ are indexed by $p$-regular partitions  $\lambda$ (that is, partitions where there are at most $p-1$ parts of a given size) and are denoted by $D^\lambda$.

\item $D^\lambda$ appears with multiplicity one as a composition factor inside any reduction mod $p$ of the corresponding ordinary irreducible representation $S^\lambda$.

\item If $\lambda$ is a $p$-core (meaning all of the hook lengths of $\lambda$ are relatively prime to $p$) then the reduction of $S^\lambda$ modulo $p$ is simple and projective.

\end{enumerate}
\end{lemma}

\noindent \emph{Second Proof of Theorem \ref{thm:main}:}  
By Lemma \ref{thm:LuoSellke} (or its predecessor in \cite{Ikenmeyer}) the tensor square $S^{\rho_n}\otimes S^{\rho_n}$ contains every $S^\lambda$ with $\lambda \ge \rho_n$ in the dominance order.  In particular this includes all $2$-regular partitions $\lambda$.  Therefore any reduction mod $2$ of $S^{\rho_n}\otimes S^{\rho_n}$ contains every irreducible $2$-modular representation $D^\lambda$ as a composition factor by Lemma \ref{thm:symmodular} part (2).
\newline \newline \noindent
By Lemma \ref{thm:symmodular} part (3), $S^{\rho_n}$ is projective in characteristic 2 since $\rho_n$ is a $2$-core. So by Lemma \ref{thm:projectives} part (2) when we tensor with the third copy of $S^{\rho_n}$ it splits all extensions and we see that $S^{\rho_n}\otimes S^{\rho_n} \otimes S^{\rho_n} $ contains a copy of $D^\lambda \otimes S^{\rho_n}$ as a direct summand for every $2$-regular partition $\lambda$.  
\newline \newline \noindent
By part $(2)$ of Lemma \ref{thm:projectives} this means that every projective indecomposable object in $Rep_2(S_N)$ is a summand of $S^{\rho_n}\otimes S^{\rho_n} \otimes S^{\rho_n} $, which lifting back to characteristic zero implies every ordinary irreducible representation occurs in $S^{\rho_n}\otimes S^{\rho_n} \otimes S^{\rho_n} $ by part (4) of Lemma \ref{thm:projectives}. \hfill $\square$

\medskip

\noindent \textbf{Remark:} A $2$-modular strengthened Saxl's conjecture was made in \cite{BBS} saying that every projective indecomposable object in $Rep_2(S_N)$ should occur as a direct summand of the Saxl square $S^{\rho_n}\otimes S^{\rho_n}$.  We'll note that this proof proves the tensor cube version of this conjecture along the way.

\bibliographystyle{alpha}
\bibliography{ref.bib}

\end{document}